\documentclass{amsart}
\usepackage{graphicx}
\usepackage{enumerate}
\usepackage{amsmath,amssymb}
\usepackage{mathrsfs}
\usepackage{color}

\newtheorem{theorem}{Theorem}[section]
\newtheorem{lemma}[theorem]{Lemma}

\theoremstyle{definition}

\newtheorem{remark}[theorem]{Remark}

\numberwithin{figure}{section}
\numberwithin{equation}{section}

\begin{document}

\title[]{Geometric Lorenz flows with historic behavior}

\author{Shin Kiriki}
\address{Department of Mathematics, Tokai University, 4-1-1 Kitakaname, 
Hiratuka Kanagawa, 259-1292, Japan.}
\email{kiriki@tokai-u.jp}

\author{Ming-Chia Li}
\address{Department of Applied Mathematics, National Chiao Tung University, Hsinchu 30050, Taiwan}
\email{mcli@math.nctu.edu.tw}

\author{Teruhiko Soma}
\address{Department of Mathematics and Information Sciences, Tokyo Metropolitan
University, Minami-Ohsawa 1-1, Hachioji, Tokyo 192-0397, Japan}
\email{tsoma@tmu.ac.jp}

\date{\today}

\begin{abstract}
We will show that, in the the geometric Lorenz flow, the set of initial states which 
give rise to orbits with historic behavior is residual in a trapping region.
\end{abstract}
\maketitle

Consider a continuous map $\varphi:X\to X$ of a compact space $X$.
We say that the forward orbit $O^+(x,\varphi)=\{x,\varphi(x),\varphi^2(x),\dots\}$ of $x\in X$ has 
\emph{historic behavior} if the 
Birkhoff average
$\lim_{n\to \infty}\frac1{n+1}\sum_{i=0}^n g(\varphi^i(x))$
does not exist for some continuous function $g:X\to \mathbb{R}$.
The notion of historic behavior was introduced by Ruelle \cite{ru}. 
We say that a subset $A$ of $X$ is a \emph{historic initial set} if, 
for any $x\in A$, the forward orbit $O^+(x,\varphi)$ has historic behavior.
Jordan, Naudot and Young \cite{jny} showed that 
the convergence of every higher order average in \cite[p.\,11]{bdv} is totally controlled by the presence of the historic initial sets.

Let $\varphi:\mathbb{S}\to \mathbb{S}$ be the doubling map on the circle 
$\mathbb{S}=\mathbb{R}/\mathbb{Z}$.
Takens \cite{ta2} showed that there exists a residual historic initial set  
in $\mathbb{S}$.
In fact, he presented only one orbit $O^+(x,\varphi)$ which is dense in $\mathbb{S}$ 
and has historic behavior.
Then, by Dowker \cite{do}, there exists a historic initial set which is residual in $\mathbb{S}$.
Dowker's theorem is very useful to show the existence of a residual historic initial set for 
various 1-dimensional maps.
The quenched random dynamics version of Takens' result is obtained by Nakano \cite{na}.
Takens' argument is applicable also to the Lorenz map $\alpha:[-1,1]\to [-1,1]$, see
Remark \ref{r_Lorenz_map}.
Many of such residual sets would have zero Lebesgue measure.
On the other hand, for any integer $r$ with $2\leq r<\infty$, Kiriki and Soma \cite{ks} proved that there exists a two-dimensional 
diffeomorphism which is arbitrarily $C^r$ close to a diffeomorphism with a quadratic homoclinic tangency 
and has a non-empty open historic initial set $D$.
Note that the open set $D$ has positive 2-dimensional Lebesgue measure.
Hence, in particular,  
this result gives an answer to Takens' Last Problem \cite{ta2} 
in the $C^2$-persistent way (see \cite[Section 6.1]{pt} for the 
definition).
Moreover, it suggests that, in certain classes of 2-dimensional diffeomorphisms, 
the historic initial set is not negligible from the physical point of view.

In this paper, we will study the historic behavior on flow dynamics.
Let $x(t)_{t\geq 0}$ be a forward orbit of a flow on a compact space $X$.
Then we say that the orbit has \emph{historic behavior} if the the time average
$$\lim_{t\to\infty}\frac1{t}\int_0^tg(x(s))\,ds$$
does not exist for some continuous function $g:X\to\mathbb{R}$.
See Takens \cite{ta1} for the definition.
Bowen's example given in \cite{ta1} is a flow on $\mathbb{R}^2$ which has a heteroclinic loop consisting 
of a pair of saddle points and two arcs connecting them.
The loop bounds an open disk $D$ in $\mathbb{R}^2$ which contains a singular point $p$ 
of the flow such that the complement $D\setminus \{p\}$ is a historic initial set.
However, this example is fragile in the sense that it is not persistent under perturbations 
which break the saddle connections.
Very recently, Labouriau and Rodrigues \cite{lr} present a persistent class of 
differential equations on $\mathbb{R}^3$ exhibiting historic behavior for an open set 
of initial conditions, which answers Takens' Last Problem for 3-dimensional flows.

Here we consider the geometric Lorenz flow introduced by Guckenheimer \cite{gu} 
as a robust model which does not belong to classes in \cite{lr}.
Robinson \cite{ro} proved that the geometric Lorenz flow is preserved under $C^2$-perturbation.
Note that Tucker \cite{tu} showed that the flow exhibited by the system of differential equations 
in Lorenz \cite{lo} (the original Lorenz flow) is realized by some geometric Lorenz model.
Our main theorem (Theorem \ref{t_1}) of this paper proves that any geometric Lorenz flow satisfying 
the conditions in Section \ref{S_Pre} has a residual historic initial set.
On the other hand, Araujo et al \cite{appv} proved that, for any singular hyperbolic attractor of a 
3-dimensional flow, the historic initial set in the topological basin of attractor has zero 
Lebesgue measure.
Since the geometric Lorenz attractor is proved to be a singular hyperbolic attractor by \cite{mpp}, 
the historic initial set is negligible from the physical point of view.
But, Theorem \ref{t_1} implies that it is not the case in dynamical systems 
from the topological point of view.

Finally, we note that Dowker's result does not work in flow dynamics.
So, in our proof, we need to construct a residual historic initial set for 
the geometric Lorentz flow practically.

\subsection*{Acknowledgements}
The authors appreciate the hospitality of NCTS, Taiwan, where parts of this work were carried out.
The first and third
authors were partially supported by JSPS KAKENHI Grant Numbers 25400112 and 26400093, respectively, and the second author by MOST 104-2115-M-009-003-MY2.

\section{Preliminaries}\label{S_Pre}
First of all, we will review the geometric Lorentz flow briefly.
See \cite{wi1,gw,wi2} for details.

Consider the square $\Sigma=\{(x, y)\in \mathbb{R}^2\,;\, |x|, |y|\leq 1\}$ and the vertical 
segment $\Gamma=\{(0, y)\in \mathbb{R}^2\,;\, |y|\leq 1\}$ in $\Sigma$.
Let $\Sigma_\pm$ be the components of $\Sigma\setminus \Gamma$ with $\Sigma_\pm\ni (\pm1, 0)$.
A map $L:\Sigma\setminus \Gamma\to \Sigma$ is said to be a \emph{Lorenz map} if it is a 
piecewise $C^2$ diffeomorphism which has the form
\begin{equation}\label{eqn_Lxy}
L(x, y) =(\alpha(x),\beta(x, y)),
\end{equation}
where $\alpha:[-1, 1]\setminus \{0\}\to [-1, 1]$ is a piecewise $C^2$-function with symmetric property
$\alpha(-x) = -\alpha(x)$ and satisfying
\begin{equation}\label{eqn_alpha}
\lim_{x\rightarrow 0+}\alpha(x) = -1,\ \alpha(1) < 1,\ 
\lim_{x\rightarrow 0+}\alpha'(x) = \infty,\ \alpha'(x) >\sqrt2
\end{equation}
for any $x \in (0, 1]$ 
(see Figure \ref{fig_1}\,(a)), and $\beta:\Sigma\setminus \Gamma\to [-1, 1]$ is a contraction in the $y$-direction.
Moreover, it is required that the images $L(\Sigma_+)$, $L(\Sigma_-)$ are mutually disjoint cusps in
$\Sigma$, where the vertices $\boldsymbol{v}_+$, $\boldsymbol{v}_-$ of $L(\Sigma_\pm)$ are contained in $\{\mp 1\}\times[-1, 1]$ respectively (see Figure \ref{fig_1}\,(b)).
\begin{figure}[hbt]
\centering
\scalebox{0.6}{\includegraphics[clip]{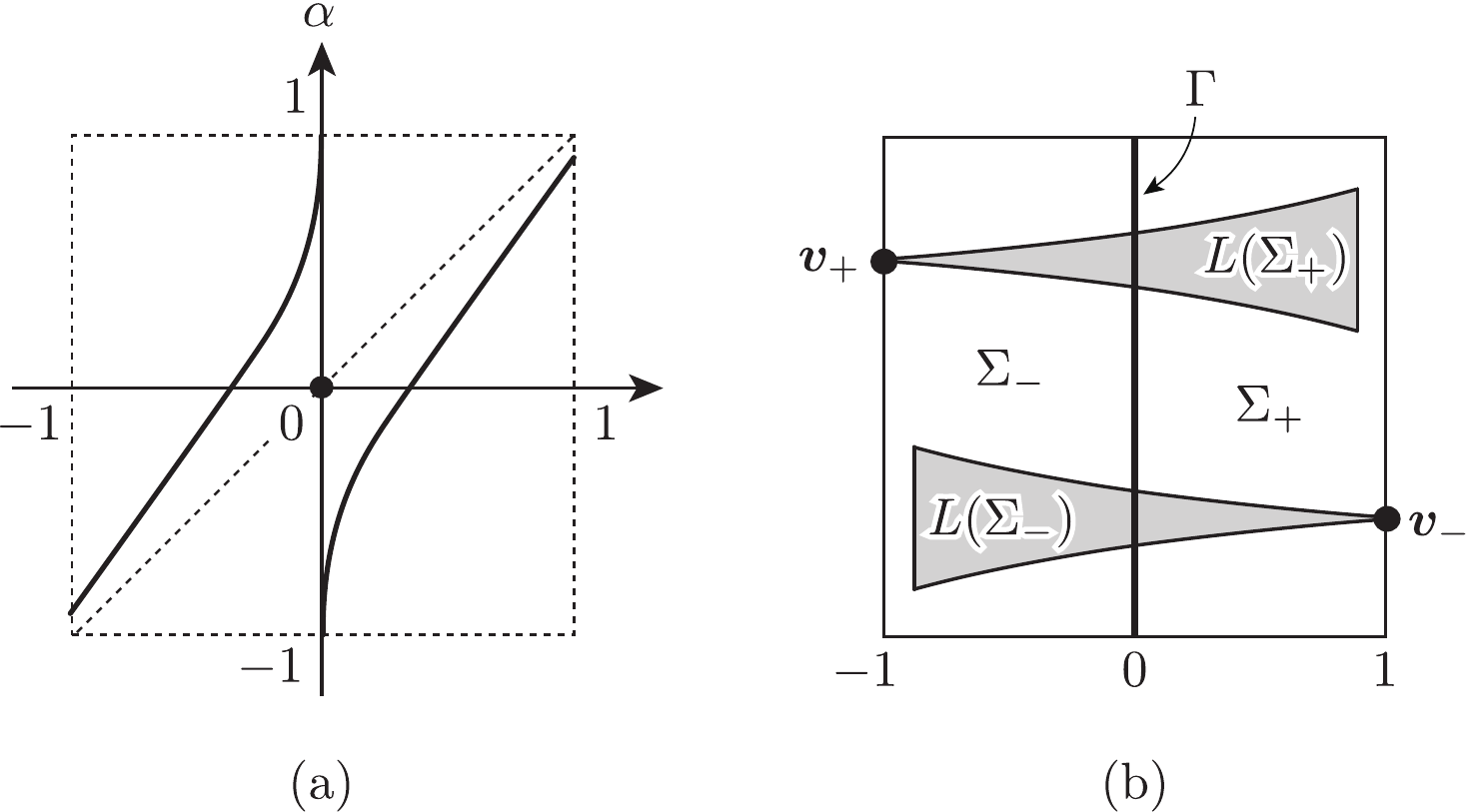}}
\caption{}
\label{fig_1}
\end{figure}

\begin{remark}[Historic behavior for the 1-dimensional Lorenz map]\label{r_Lorenz_map}
We denote the forward orbit of $x\in [-1,1]$ under $\alpha$ by $O^+(x,\alpha)$.
By Hofbauer \cite{ho}, the dynamics of $\alpha$ on $[-1,1]$ is described by a Markov partition on finite 
symbols.
Let $s'$ be a periodic sequence of these symbols and $s''$ a sequence such that, for the point $x''$ of $[-1,1]$ 
corresponding to $s''$, the partial averages $\frac1{n+1}\sum_{i=0}^n\delta_{\alpha^i(x'')}$ converge 
to the Lebesgue measure.
As in Takens \cite[Section 4]{ta1}, there exists a sequence $s_0$ of these symbols in which long initial 
segments of $s'$ and those of $s''$ appear alternately and such that, for the point $x_0$ of $I$ 
corresponding to $s_0$, $O^+(x_0,\alpha)$ is dense in $[-1,1]$ and has historic behavior.
Then, by Dowker \cite{do}, there exists a historic initial set which is residual in $[-1,1]$.
\end{remark}

We identify the square $\Sigma$ and any subset of $\Sigma$ with their images 
in $\mathbb{R}^3$ via the embedding $\iota:\mathbb{R}^2\to \mathbb{R}^3$ with $\iota(x,y)=(x,y,1)$.
A $C^2$-vector field $X_L$ on $\mathbb{R}^3$ is said to be a \emph{geometric Lorenz vector 
field controlled by} the Lorenz map $L:\Sigma\setminus \Gamma\to \Sigma$ (\ref{eqn_Lxy}) if it satisfies the 
following conditions (i) and (ii).
\begin{enumerate}[(i)]
\item
For any point $(x, y, z)$ in a neighborhood of the origin $\mathbf{0}$ of $\mathbb{R}^3$, 
$X_L$ is given by the differential equation
\begin{equation}\label{eqn_Lf}
\dot x=\lambda x,\quad \dot y=-\mu y,\quad \dot z=-\nu z
\end{equation}
for some $\lambda>0$, $\mu>\nu>0$.
Moreover, $\Gamma$ is contained in the stable manifold $W^s(\mathbf{0})$ of $\mathbf{0}$.
\item
All forward orbits of $X$ starting from $\Sigma\setminus \Gamma$ will return to 
$\Sigma$ and the first return map is $L$.
\end{enumerate}
Note then that $\mathbf{0}$ is a singular point (an equilibrium) of saddle type, the local
unstable manifold of $\mathbf{0}$ is tangent to the $x$-axis, and the local stable manifold 
of $\mathbf{0}$ is tangent to the $yz$-plane, see Figure \ref{fig_2}. 
\begin{figure}[hbt]
\centering
\scalebox{0.6}{\includegraphics[clip]{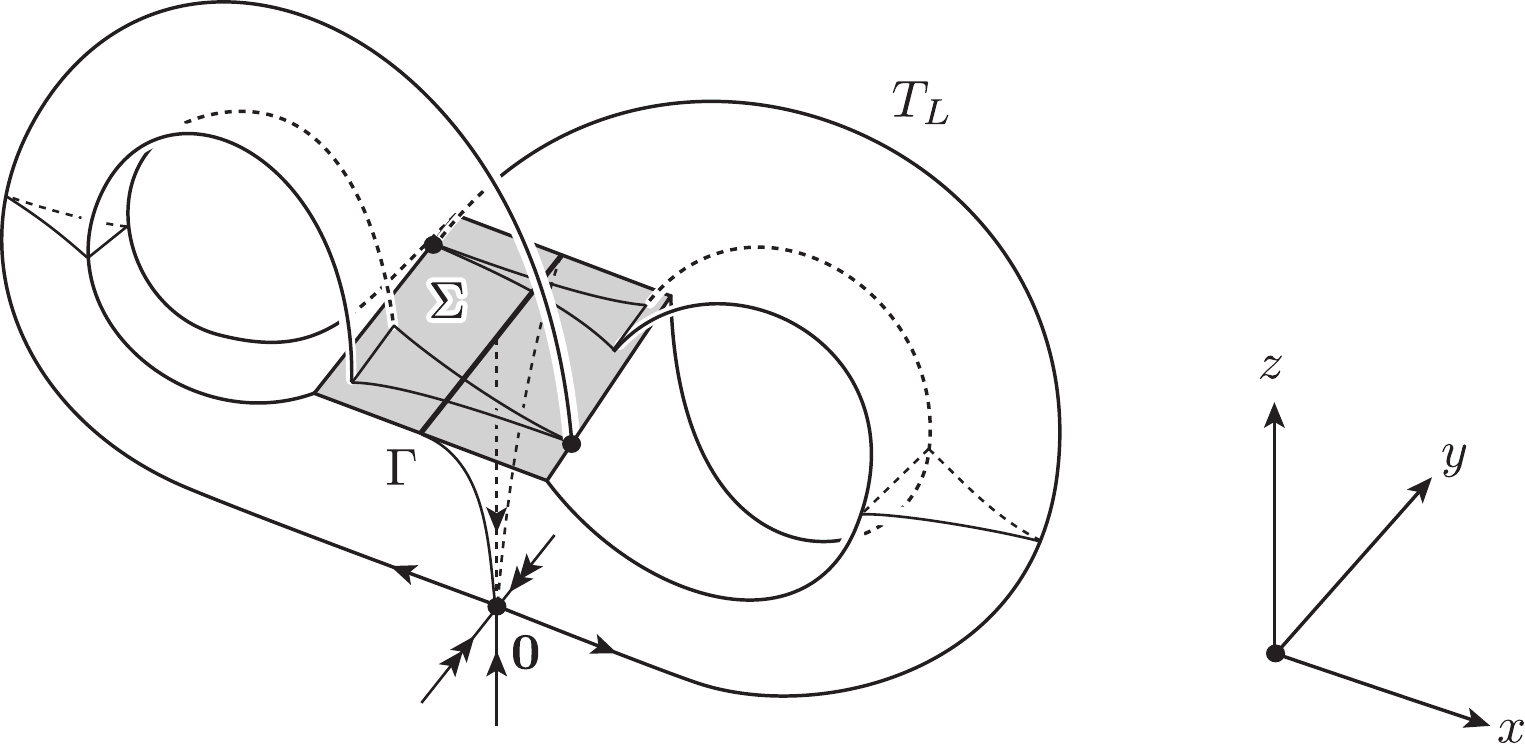}}
\caption{}
\label{fig_2}
\end{figure}
The $C^2$-map $\varphi_L :\mathbb{R}^3\times \mathbb{R}\to \mathbb{R}^3$ 
defined by $\varphi_L(\boldsymbol{x}, 0) = \boldsymbol{x}$ and 
$(\partial/\partial t)\varphi_L(\boldsymbol{x}, t) = X_L(\varphi_L(\boldsymbol{x}, t))$
is called the \emph{geometric Lorenz flow} associated with the vector field $X_L$.
The closure of $\bigcup_{\boldsymbol{z}\in \Sigma\setminus\Gamma}\varphi_L(\boldsymbol{z}, [0,\infty))$ in 
$\mathbb{R}^3$ is homeomorphic to a genus two handlebody as
illustrated in Figure \ref{fig_2}, which is called the \emph{trapping region} of 
$\varphi_L$ and denoted by $T_{\varphi_L}$ or $T_L$. 
Any forward orbit for $\varphi_L$ with its initial point in $T_L$ cannot escape from $T_L$.

For simplicity, we suppose moreover that the geometric Lorenz flow satisfies the differential equation 
(\ref{eqn_Lf}) on 
$$V_L=T_L\cap [-0.1,0.1]\times [-1,1]\times [0,1].$$ 
In fact, this assumption is not crucial and our subsequent argument still works for any geometric Lorenz flow which satisfies (\ref{eqn_Lf}) only on an arbitrarily small neighborhood of $\mathbf{0}$ in $T_L$.

\section{Historic behavior for the geometric Lorenz flow}\label{S_L_flow}

Let $\varphi_L$ be the geometric Lorenz flow given in the previous section.
Suppose that $g:T_L\to \mathbb{R}$ is a continuous function on the trapping region $T_L$. 
For $\tau>0$ and $\delta>0$, the forward orbit $\varphi_L(\boldsymbol{x},t)_{t\geq 0}$ emanating from $\boldsymbol{x}\in T_L$ 
is said to have $(\tau,\delta)$\emph{-historic 
behavior} with respect to $g$ if there exist $\tau_0$, $\tau_1$ with $\tau_0,\tau_1\geq \tau$ such that
$$\left|\frac1{\tau_0}\int_0^{\tau_0}g(\varphi_L(\boldsymbol{x},t))\,dt-\frac1{\tau_1}\int_0^{\tau_1}g(\varphi_L(\boldsymbol{x},t))\,dt \,\right|\geq \delta.$$
In particular, $\varphi_L(\boldsymbol{x},t)_{t\geq 0}$ has historic behavior if and only if 
there exists $\delta>0$ and a continuous function $g$ on $T_L$ such that, for any $\tau>0$, 
$\varphi_L(\boldsymbol{x},t)_{t\geq 0}$ has $(\tau,\delta)$-historic behavior with respect to $g$.

For any $\boldsymbol{y},\boldsymbol{z}\in T_L$ contained in the same forward orbit $\varphi_L(\boldsymbol{x},[0,\infty))$ with 
$\boldsymbol{x}\in\Sigma$, 
the sub-arc of $\varphi_L(\boldsymbol{x},[0,\infty))$ connecting $\boldsymbol{y}$ with $\boldsymbol{z}$ is denoted by $\Phi_L(\boldsymbol{y},\boldsymbol{z})$ 
or $\Phi_L(\boldsymbol{z},\boldsymbol{y})$.
Let $t_{\boldsymbol{x}}(\boldsymbol{y})\geq 0$ be the number with $\varphi_L(\boldsymbol{x},t_{\boldsymbol{x}}(\boldsymbol{y}))=\boldsymbol{y}$.
We set $\tau(\boldsymbol{y},\boldsymbol{z})=|t_{\boldsymbol{x}}(\boldsymbol{y})-t_{\boldsymbol{x}}(\boldsymbol{z})|$.
Note that $\tau(\boldsymbol{y},\boldsymbol{z})$ is independent of $\boldsymbol{x}\in \Sigma$ with $\varphi_L(\boldsymbol{x},[0,\infty))\ni \boldsymbol{y},\boldsymbol{z}$.
We also set $\tau(\boldsymbol{y},\boldsymbol{z})=\tau(\gamma)$ if $\gamma=\Phi_L(\boldsymbol{y},\boldsymbol{z})$.
Let $A$ be a compact subset of $T_L\setminus \{\mathbf{0}\}$ such that $\Phi_L(\boldsymbol{y},\boldsymbol{z})\cap A$ is a disjoint union of finitely many arcs $\gamma_1,\dots,\gamma_n$.
Then the total sum $\sum_{i=1}^n \tau(\gamma_i)$ is denoted by $\tau(\boldsymbol{y},\boldsymbol{z})|_A$.

Take a periodic point $x_{\mathrm{per}(2)}$ of $\alpha$ with period two. 
Let $\pi:\mathbb{R}^3\to \mathbb{R}^2$ be the orthogonal projection defined by $\pi(x,y,z)=(x,z)$.
For any point $\boldsymbol{x}$ of $\Sigma$ with $\boldsymbol{x}_{[1]}=
x_{\mathrm{per}(2)}$, 
the the image $Q(x_{\mathrm{per}(2)})=\pi(\varphi_L(\boldsymbol{x},[0,\infty)))$ is a closed curve in the $xz$-plane 
disjoint from the origin of $\mathbb{R}^2$.
Here we denote the first entry of an element $\boldsymbol{a}$ of $\mathbb{R}^3$ by $\boldsymbol{a}_{[1]}$, that is, 
$(a,b,c)_{[1]}=a$.
Though $Q(x_{\mathrm{per}(2)})$ depends on $x_{\mathrm{per}(2)}$, it is 
independent of the $y$-entry of $\boldsymbol{x}$.
See Figure \ref{fig_3}.
\begin{figure}[hbt]
\centering
\scalebox{0.6}{\includegraphics[clip]{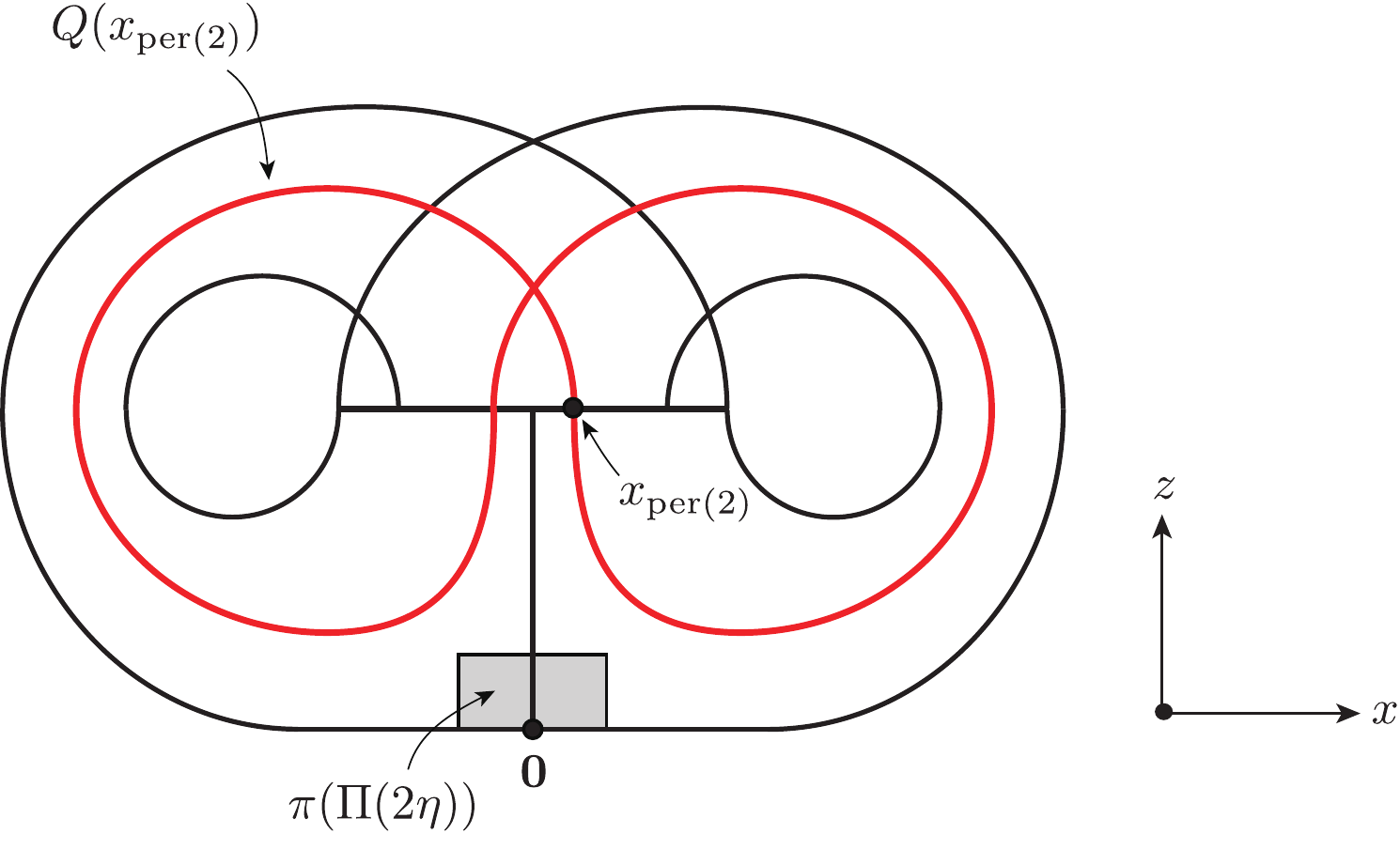}}
\caption{}
\label{fig_3}
\end{figure}

For $\eta>0$, the rectangular solid $[-\eta,\eta]^2\times [0,\eta]$ is denoted by $\Pi(\eta)$.
By taking the $\eta$ sufficiently small, one can suppose that 
$\Pi(\eta)\cap T_L\subset V_L$ and 
 $\Pi(2\eta)\cap \varphi_L(\boldsymbol{x},[0,\infty))=\emptyset$ for any $\boldsymbol{x}\in \Sigma$ with 
$\boldsymbol{x}_{[1]}=x_{\mathrm{per}(2)}$.
Consider the subspaces 
$$\partial_{\mathrm{side}}\Pi(\eta)=
\{-\eta,\eta\}\times [-\eta,\eta]\times [0,\eta]\quad\text{and}\quad
\partial_{\mathrm{top}}\Pi(\eta)=
[-\eta,\eta]^2\times \{\eta\}$$
of the boundary $\partial \Pi(\eta)$.

By the third equation of (\ref{eqn_Lf}) on $V_L$, any sub-arc $\gamma$ of an orbit 
connecting $\Sigma$ with $\partial_{\mathrm{top}}\Pi(\eta)$ in 
$T_L\cap [-\eta,\eta]\times [-1,1]\times [\eta,1]\subset V_L$ satisfies
\begin{equation}\label{eqn_tau_C'}
\tau(\gamma)=\frac{\log(\eta^{-1})}{\nu}.
\end{equation}

Note that the Lorenz flow does not have singular points in the compact set $\overline{T_L\setminus \Pi(\eta)}$, where $\overline A$ denotes the closure of a subset $A$ of $T_L$.
It follows from the fact that there exists a constant $C>0$ satisfying  
\begin{equation}\label{eqn_tau}
\tau(\boldsymbol{x},L(\boldsymbol{x}))|_{\,\overline{T_L\setminus \Pi(\eta)}}\leq C
\end{equation}
for any $\boldsymbol{x}\in\Sigma\setminus \Gamma$.

\medskip

The following is our main theorem in this paper.

\begin{theorem}\label{t_1}
There exists a residual subset $\mathcal{H}$ of $\Sigma$ such that, 
for any $\boldsymbol{x}\in \mathcal{H}$, the forward orbit $\varphi_L(\boldsymbol{x},t)_{t\geq 0}$ 
has historic behavior. 
\end{theorem}

Here we fix a continuous function $g:T_L\to \mathbb{R}$ satisfying the following condition.
\begin{enumerate}[(1)]
\item
$0\leq g(\boldsymbol{x})\leq 1$ for any $\boldsymbol{x}\in T_L$.
\item
The support of $g$ is contained in $\Pi(2\eta)\cap T_L$ and 
$g(\boldsymbol{x})=1$ on $\Pi(\eta)\cap T_L$.
\end{enumerate}

\medskip

The following lemma is crucial in the proof of Theorem \ref{t_1}.

\begin{lemma}\label{l_1}
For any positive integer $N$, any $0<\varepsilon<1$ and $\boldsymbol{x}_0\in \Sigma$, 
there exists an open disk $U_{(\boldsymbol{x}_0,N,\varepsilon)}$ contained in the $\varepsilon$-neighborhood of $\boldsymbol{x}_0$ in $\Sigma$ and satisfying the following condition.
\begin{enumerate}[{\rm (H$_N$)}]
\item
For any $\boldsymbol{z}\in U_{(\boldsymbol{x}_0,N,\varepsilon)}$, $\varphi_L(\boldsymbol{z},t)_{t\geq 0}$ has $(N,1/2)$-historic behavior with respect to $g$.
\end{enumerate}
\end{lemma}

Here we note that the disk $U_{(\boldsymbol{x}_0,N,\varepsilon)}$ is not necessarily required to have $x_0$ 
as an element.

\begin{proof}
Since $0<\varepsilon<1$, $\lim_{\sigma\rightarrow\infty}\log (\varepsilon^{-\sigma})=\infty$.
Thus one can have a constant $\sigma\geq 1$ satisfying 
\begin{equation}\label{eqn_loge}
\frac1{\lambda}\bigl(\log (\varepsilon^{-\sigma})+\log \eta\bigr)\geq \max\left\{N,\frac{6C\log(\varepsilon^{-1})}{\log 2}+\frac{4\log(\eta^{-1})}{\nu}\right\}.
\end{equation}

Set $(\boldsymbol{x}_0)_{[1]}=x_0$ and consider the interval $I(x_0,\varepsilon)=[x_0-\varepsilon,x_0+\varepsilon]$ in 
the $x$-axis.
Let $n_0\in\mathbb{N}$ be the smallest non-negative integer such that $\alpha^{n_0}(I(x_0,\varepsilon))$ contains $0$.
We denote the length of an interval $I$ in the $x$-axis by $\ell(I)$.
Since $\ell(\alpha^{n_0}(I(x_0,\varepsilon))\leq 2$ and $|\alpha'(x)|> \sqrt2$ for any $x\in [-1,1]$ by (\ref{eqn_alpha}), we have $2\varepsilon(\sqrt2)^{n_0}\leq 2$ or equivalently 
\begin{equation}\label{eqn_n_geq}
n_0\leq \frac{2\log (\varepsilon^{-1})}{\log 2}.
\end{equation}

On the other hand, since $\ell(\alpha^{n_0}(I(x_0,\varepsilon)))$ is at least $2\varepsilon$, $\alpha^{n_0}(I(x_0,\varepsilon))$ contains 
either $[\varepsilon^\sigma/2,\varepsilon^\sigma]$ or $[-\varepsilon^\sigma,-\varepsilon^\sigma/2]$, say $[\varepsilon^\sigma/2,\varepsilon^\sigma]$.
Then $I(x_0,\varepsilon)$ contains an interval $J_0$ with $\alpha^{n_0}(J_0)=[\varepsilon^\sigma/2,\varepsilon^\sigma]$.
For any $\boldsymbol{y}_0\in \Sigma$ with $(\boldsymbol{y}_0)_{[1]}\in J_0$, set $\boldsymbol{y}_1=L^{n_0}(\boldsymbol{y}_0)$.
Let $\boldsymbol{y}_3$ be the first point in  
$\varphi_L(\boldsymbol{y}_1,[0,\infty))$ meeting $\partial_{\mathrm{side}}\Pi(\eta)$.
From our setting of these points, we have a unique intersection point $\boldsymbol{y}_2$ of $\Phi_L(\boldsymbol{y}_1,\boldsymbol{y}_3)$ with $\partial_{\mathrm{top}}\Pi(\eta)$.
Then
\begin{equation}\label{eqn_y0y3}
\begin{split}
\tau(\boldsymbol{y}_0,\boldsymbol{y}_3)&=\tau(\boldsymbol{y}_0,\boldsymbol{y}_1)+\tau(\boldsymbol{y}_1,\boldsymbol{y}_2)+\tau(\boldsymbol{y}_2,\boldsymbol{y}_3)\\
&=\tau(\boldsymbol{y}_0,\boldsymbol{y}_1)|_{T_L\cap \Pi(\eta)}+\tau(\boldsymbol{y}_0,\boldsymbol{y}_1)|_{\,\overline{T_L\setminus \Pi(\eta)}} +\tau(\boldsymbol{y}_1,\boldsymbol{y}_2)+\tau(\boldsymbol{y}_2,\boldsymbol{y}_3).
\end{split}
\end{equation}
By (\ref{eqn_tau_C'}),
$$\tau(\boldsymbol{y}_1,\boldsymbol{y}_2)=\frac{\log(\eta^{-1})}{\nu}.$$
By (\ref{eqn_tau}) and (\ref{eqn_n_geq}), 
$$\tau(\boldsymbol{y}_0,\boldsymbol{y}_1)|_{\,\overline{T_L\setminus \Pi(\eta)}}\leq \frac{2C\log (\varepsilon^{-1})}{\log 2}.$$
Since $\varepsilon^\sigma/2\leq (\boldsymbol{y}_1)_{[1]}\leq \varepsilon^\sigma$ and $(\boldsymbol{y}_3)_{[1]}=\eta$, 
it follows from the first equation of (\ref{eqn_Lf}) and (\ref{eqn_loge}) that
\begin{equation}\label{eqn_tauy1y3}
\tau(\boldsymbol{y}_1,\boldsymbol{y}_3)\geq \frac1{\lambda}\bigl(\log (\varepsilon^{-\sigma})+\log \eta\bigr)\geq N.
\end{equation}
By the former inequality of (\ref{eqn_tauy1y3}) together with (\ref{eqn_loge}),
\begin{equation}\label{eqn_tauy2y3}
\begin{split}
\tau(\boldsymbol{y}_2,\boldsymbol{y}_3)&=\tau(\boldsymbol{y}_1,\boldsymbol{y}_3)-\tau(\boldsymbol{y}_1,\boldsymbol{y}_2)\geq 
\frac1{\lambda}\bigl(\log (\varepsilon^{-\sigma})+\log \eta\bigr)-\frac{\log(\eta^{-1})}{\nu}\\
&\geq 3\left(\frac{2C\log(\varepsilon^{-1})}{\log 2}+\frac{\log(\eta^{-1})}{\nu}\right)\\
&\geq 3\bigl(\tau(\boldsymbol{y}_0,\boldsymbol{y}_1)_{\,\overline{T_L\setminus \Pi(\eta)}}+\tau(\boldsymbol{y}_1,\boldsymbol{y}_2)\bigr).
\end{split}
\end{equation}
Since $\Phi_L(\boldsymbol{y}_1,\boldsymbol{y}_2)\cap T_L\cap \Pi(\eta)$ is a single point of $\partial_{\mathrm{top}}\Pi(\eta)$, we have
$\tau(\boldsymbol{y}_1,\boldsymbol{y}_2)|_{T_L\cap \Pi(\eta)}=0$.
Since $\Phi_L(\boldsymbol{y}_2,\boldsymbol{y}_3)\subset T_L\cap \Pi(\eta)$, 
$\tau(\boldsymbol{y}_2,\boldsymbol{y}_3)|_{T_L\cap \Pi(\eta)}=\tau(\boldsymbol{y}_2,\boldsymbol{y}_3)$.
This shows that 
$$\tau(\boldsymbol{y}_1,\boldsymbol{y}_3)|_{T_L\cap \Pi(\eta)}=\tau(\boldsymbol{y}_1,\boldsymbol{y}_2)|_{T_L\cap \Pi(\eta)}+\tau(\boldsymbol{y}_2,\boldsymbol{y}_3)|_{T_L\cap \Pi(\eta)}=\tau(\boldsymbol{y}_2,\boldsymbol{y}_3).$$
It follows from (\ref{eqn_y0y3}) that
\begin{align*}
\frac{\tau(\boldsymbol{y}_0,\boldsymbol{y}_3)}{\tau(\boldsymbol{y}_0,\boldsymbol{y}_3)|_{T_L\cap \Pi(\eta)}}
&=\frac{\tau(\boldsymbol{y}_0,\boldsymbol{y}_3)}{\tau(\boldsymbol{y}_0,\boldsymbol{y}_1)|_{T_L\cap \Pi(\eta)}+\tau(\boldsymbol{y}_2,\boldsymbol{y}_3)}\\
&=1+\frac{\tau(\boldsymbol{y}_0,\boldsymbol{y}_1)|_{\,\overline{T_L\setminus \Pi(\eta)}}+\tau(\boldsymbol{y}_1,\boldsymbol{y}_2)}{\tau(\boldsymbol{y}_0,\boldsymbol{y}_1)|_{T_L\cap \Pi(\eta)}+\tau(\boldsymbol{y}_2,\boldsymbol{y}_3)}\\
&\leq 1+\frac{\tau(\boldsymbol{y}_0,\boldsymbol{y}_1)|_{\,\overline{T_L\setminus \Pi(\eta)}}+\tau(\boldsymbol{y}_1,\boldsymbol{y}_2)}{\tau(\boldsymbol{y}_2,\boldsymbol{y}_3)}\leq \frac43.
\end{align*}
Since $g$ is a non-negative continuous function with $g(\boldsymbol{x})=1$ on $T_L\cap \Pi(\eta)$,
\begin{equation}\label{eqn_3/4}
\frac1{\tau(\boldsymbol{y}_0,\boldsymbol{y}_3)}\int_0^{\tau(\boldsymbol{y}_0,\boldsymbol{y}_3)}
g(\varphi_L(\boldsymbol{y}_0,t))\,dt\geq \frac{\tau(\boldsymbol{y}_0,\boldsymbol{y}_3)|_{T_L\cap \Pi(\eta)}}{\tau(\boldsymbol{y}_0,\boldsymbol{y}_3)}\geq \frac34.
\end{equation}

Since $\bigcup_{n=1}^\infty\alpha^{-n}(x_{\mathrm{per}(2)})$ is dense in $[-1,1]$, 
there exists an $n_1\in \mathbb{N}$ such that the interior of $\alpha^{n_1}([\varepsilon^\sigma/2,\varepsilon^\sigma])
=\alpha^{n_0+n_1}(J_0)$ contains $x_{\mathrm{per}(2)}$.
There exists a closed subinterval $J_1$ of $[\varepsilon^\sigma/2,\varepsilon^\sigma]$ such that $\mathrm{Int}\,\alpha^{n_1}(J_1)\ni x_{\mathrm{per}(2)}$ 
and $\bigcup_{i=1}^{n_1}\alpha^i(J_1)\cap \{0\}=\emptyset$.
There exists a point $\boldsymbol{y}_0$ in the interior of the $\varepsilon$-neighborhood of $\boldsymbol{x}_0$ 
such that the points $\boldsymbol{y}_1=L^{n_0}(\boldsymbol{y}_0)$ and $\boldsymbol{y}_4=L^{n_0+n_1}(\boldsymbol{y}_0)$ satisfy 
$(\boldsymbol{y}_1)_{[1]}\in \mathrm{Int}J_1$ and $(\boldsymbol{y}_4)_{[1]}=x_{\mathrm{per}(2)}$ respectively.
Take a point $\boldsymbol{y}_5$ in $\varphi_L(\boldsymbol{y}_4,t)_{t\geq 0}$ with 
$$\tau(\boldsymbol{y}_0,\boldsymbol{y}_5)\geq 5\tau(\boldsymbol{y}_0,\boldsymbol{y}_4).$$
Since $\psi_L(\boldsymbol{y}_4,[0,\infty))\cap \Pi(2\eta)=\emptyset$, $g(\boldsymbol{x})=0$ on $\Phi_L(\boldsymbol{y}_4,\boldsymbol{y}_5)$.
Since moreover $0\leq g(\boldsymbol{x})\leq 1$ on $T_L$, we have 
$$\frac1{\tau(\boldsymbol{y}_0,\boldsymbol{y}_5)}\int_0^{\tau(\boldsymbol{y}_0,\boldsymbol{y}_5)}g(\psi_L(\boldsymbol{y}_0,t))\,dt
\leq \frac{\tau(\boldsymbol{y}_0,\boldsymbol{y}_4)}{\tau(\boldsymbol{y}_0,\boldsymbol{y}_5)}\leq \frac15.$$
By this inequality together with (\ref{eqn_3/4}), we have
\begin{align*}
\frac1{\tau(\boldsymbol{y}_0,\boldsymbol{y}_3)}\int_0^{\tau(\boldsymbol{y}_0,\boldsymbol{y}_3)}& g(\psi_L(\boldsymbol{y}_0,t))\,dt-
\frac1{\tau(\boldsymbol{y}_0,\boldsymbol{y}_5)}\int_0^{\tau(\boldsymbol{y}_0,\boldsymbol{y}_5)}g(\psi_L(\boldsymbol{y}_0,t))\,dt
\\
&\geq \frac34-\frac15=\frac{11}{20}.
\end{align*}
Then one can have a small open disk $U(\boldsymbol{x}_0,N,\varepsilon)$ centered at $\boldsymbol{y}_0$ and contained in the 
$\varepsilon$-neighborhood of $\boldsymbol{x}_0$ such that, for any $\boldsymbol{z}\in U(\boldsymbol{x}_0,N,\varepsilon)$, 
$$
\frac1{\tau(\boldsymbol{y}_0,\boldsymbol{y}_3)}\int_0^{\tau(\boldsymbol{y}_0,\boldsymbol{y}_3)} g(\psi_L(\boldsymbol{z},t))\,dt-
\frac1{\tau(\boldsymbol{y}_0,\boldsymbol{y}_5)}\int_0^{\tau(\boldsymbol{y}_0,\boldsymbol{y}_5)}g(\psi_L(\boldsymbol{z},t))\,dt
\geq \frac12.
$$
By (\ref{eqn_tauy1y3}), we also have 
$$\tau(\boldsymbol{y}_0,\boldsymbol{y}_5)>\tau(\boldsymbol{y}_0,\boldsymbol{y}_3)>\tau(\boldsymbol{y}_1,\boldsymbol{y}_3)\geq N.$$
It follows that $\psi_L(\boldsymbol{z},t)_{t\geq 0}$ has $(N,1/2)$-historic behavior 
with respect to $g$.
\end{proof}

\begin{proof}[Proof of Theorem \ref{t_1}]
For any $N,m\in \mathbb{N}$ and any $\boldsymbol{x}\in \Sigma\setminus \Gamma$,  
let $U_{(\boldsymbol{x},N,1/(m+1))}$ be 
the open disk given in Lemma \ref{l_1} with $\varepsilon=1/(m+1)$.
Then the union $\mathcal{U}_N=\bigcup_{m\in \mathbb{N},\boldsymbol{x}\in \Sigma\setminus \Gamma}U_{(\boldsymbol{x},N,1/(m+1))}$ is an open dense subset of $\Sigma$, and hence 
$\mathcal{H}=\bigcap_{N=1}^\infty \mathcal{U}_N$ is a residual subset of $\Sigma$.
Since each element $\boldsymbol{z}$ of $\mathcal{H}$ satisfies the condition (H$_N$) of Lemma \ref{l_1} 
for any $N\in \mathbb{N}$, 
the forward orbit $\varphi_L(\boldsymbol{z},t)_{t\geq 0}$ has historic behavior.
This completes the proof.
\end{proof}

\end{document}